\documentclass[12pt]{article}

\usepackage[utf8]{inputenc}
\usepackage[a4paper,nomarginpar]{geometry}
\usepackage[english]{babel}
\usepackage{enumitem}
\usepackage{amsmath,amsfonts,amssymb,amsthm,amscd}
\usepackage{tikz,xcolor}
\usepackage{mathrsfs}
\allowdisplaybreaks

\theoremstyle{plain}
\newtheorem{theorem}{Theorem}
\newtheorem{lemma}[theorem]{Lemma}
\newtheorem*{theoremA}{Theorem A}

\theoremstyle{definition}
\newtheorem{remark}[theorem]{Remark}

\newcommand{\N}{\mathbb{N}} 
\newcommand{\C}{\mathbb{C}} 
\newcommand{\Nzero}{\mathbb N_0}

\newcommand{\Hcal}{\mathcal H}
\newcommand{\Lop}{\mathcal L}
\DeclareMathOperator{\dist}{dist}

\newcommand{\eps}{\varepsilon}

\title{The endpoint problem for $\varepsilon$-hypercyclicity}

\author{Vinícius V. Fávaro, Daniel Pellegrino, Jo\~ao V. A. Pinto, Geivison Ribeiro}

\date{}

\begin{document}
	
	\maketitle
	
	\begin{abstract}
		For a fixed $0<\varepsilon<1$, F. Bayart asked in 2024 whether there exists an operator $T$ such that, for every $0<\delta<1$, $T$ is $\delta$-hypercyclic if and only if $\delta\in[\varepsilon,1)$. We answer this question affirmatively by constructing a weighted backward shift on $\ell_2(\mathbb N_0,\ell_2(\mathbb N_0))$ with this property.
	\end{abstract}
	
	\bigskip\noindent
	{\bf Keywords:} $\varepsilon$-hypercyclicity,
	weighted shifts, linear dynamics, operator theory.
	
	\bigskip\noindent
	{\bf 2020 Mathematics Subject Classification:} Primary 47A16; Secondary 47B37.
	\section{Introduction}
	Let $X$ be a separable infinite-dimensional Banach space and let $T:X\to X$ be a continuous linear operator. We say that $T$ is \emph{hypercyclic} if there exists $x\in X$ whose orbit
	\[
	\operatorname{Orb}(T,x):=\{T^nx:n\in\Nzero\}
	\]
	is dense in $X$. We refer to \cite{BayartMatheron,GrosseErdmannPeris} for the basic theory of hypercyclic operators.
	
	A natural question in linear dynamics is whether an orbit satisfying a weaker approximation property must already be dense. In this direction, Badea, Grivaux and Müller introduced in \cite{BadeaGrivauxMuller} the notion of an $\varepsilon$-hypercyclic operator.
	
	For $0<\varepsilon<1$, a vector $x\in X$ is called \emph{$\varepsilon$-hypercyclic} for $T$ if, for every $y\in X\setminus\{0\}$, there exists $n\in\Nzero$ such that
	\[
	\|T^nx-y\|\leq\varepsilon\|y\|.
	\]
	The operator $T$ is $\varepsilon$-hypercyclic if it admits an $\varepsilon$-hypercyclic vector. Clearly, every hypercyclic operator is $\varepsilon$-hypercyclic for every $0<\varepsilon<1$. The converse is false: Badea, Grivaux and Müller \cite{BadeaGrivauxMuller} constructed examples on $\ell_1$ and Bayart \cite{BayartHilbert} obtained an example on a separable Hilbert space. 
	
	The notion is monotone in the parameter: if $T$ is $\varepsilon$-hypercyclic, then it is $\delta$-hypercyclic for every $\varepsilon\leq\delta<1$. It is therefore natural to ask which intervals may occur as
	\[
	\{\delta\in(0,1):T\text{ is }\delta\text{-hypercyclic}\}.
	\]
	Bayart proved in \cite{BayartThreshold} that, for every $0<\varepsilon<1$, there exists a bounded operator on a separable Hilbert space which is $\delta$-hypercyclic for every $\delta>\varepsilon$ and is not $\delta$-hypercyclic for any $\delta<\varepsilon$. The endpoint $\delta=\varepsilon$ was left open (see \cite[Question~2.6]{BayartThreshold}). In this paper we prove that the endpoint can also be attained:
	
	\begin{theoremA}\label{thm:main}
		For every $0<\varepsilon<1$, there exists a bounded linear operator $T$ on a separable infinite-dimensional Hilbert space such that, for every $0<\delta<1$,
		\[
		T\text{ is }\delta\text{-hypercyclic}
		\quad\Longleftrightarrow\quad
		\delta\in[\varepsilon,1).
		\]
	\end{theoremA}
	
	We will construct a weighted backward shift on $\ell_2(\Nzero,\ell_2(\Nzero))$. The main idea is to concatenate finite blocks of weights. Their partial products give the lower estimates needed to exclude every $\delta<\varepsilon$, while the lengths of the blocks are chosen recursively so that a countable family of targets can be encoded without losing the endpoint estimate.
	
	The paper is organized as follows. In Section~\ref{sec:preliminaries}, we fix the notation and prove an elementary Hilbert space estimate used in the lower bound. In Section~\ref{sec:construction}, we construct the finite blocks, obtain the coding estimates and concatenate the weights. Finally, in Section~\ref{sec:main-result}, we define the backward shift and prove Theorem~A.
	
	\section{Preliminaries}\label{sec:preliminaries}
	
	Let $X$ be a Banach space. We denote by $\Lop(X)$ the space of all bounded linear operators on $X$. We also denote by $I_X$ the identity operator. Throughout this paper, all Banach and Hilbert spaces are assumed to be complex. $\mathbb{N}$ denotes the set of all positive integers, and $\mathbb{N}_0:=\mathbb{N}\cup\{0\}$. We denote by $\ell_2(\N_0)$ the Hilbert space of all scalar sequences $x=(x_n)_{n\in\N_0}$ such that $\sum_{n=0}^{\infty}|x_n|^2<\infty$ and write $\|\cdot\|_2$ for its norm. We denote the canonical orthonormal basis of $\ell_2(\Nzero)$ by $(e_n)_{n\in\Nzero}$. If $H$ is a Hilbert space, then we denote its inner product and associated norm by $\langle\cdot,\cdot\rangle_H$ and $\|\cdot\|_H$, respectively. As usual,
	\[
	\dist(x,M):=\inf_{y\in M}\|x-y\|
	\]
	for every subset $M$ of a normed space and every vector $x$ in that space.
	
	For a Hilbert space $H$, let
	\[
	\ell_2(\Nzero,H)
	:=
	\left\{(x_n)_{n\in\Nzero}\in H^{\Nzero}:
	\sum_{n=0}^{\infty}\|x_n\|_H^2<\infty
	\right\},
	\]
	endowed with its natural Hilbert space structure. The space used in our construction will be
	\[
	\Hcal:=\ell_2(\Nzero,\ell_2(\Nzero))
	\]
	and its norm is denoted simply by $\|\cdot\|$. Thus, if $x=(x_n)_{n\in\Nzero}\in\Hcal$, then
	\[
	\|x\|^2=\sum_{n=0}^{\infty}\|x_n\|_2^2.
	\]
	
	We begin with an elementary estimate that will be used in the proof of our main result.
	
	\begin{lemma}\label{lem:hilbert-estimate}
		Let $H$ be a Hilbert space and let $x,y\in H$ be orthonormal vectors. Then, for every $t>0$,
		\[
		\min_{z\in\C}\|(z-1)x+zty\|_H
		=
		\frac{t}{\sqrt{1+t^2}}.
		\]
		Moreover, this minimum is attained uniquely at
		\[
		z=\frac{1}{1+t^2}.
		\]
	\end{lemma}
	
	\begin{proof}
		Fix $t>0$. Since $x$ and $y$ are orthonormal, for every $z\in\C$, the vectors $(z-1)x$ and $zty$ are orthogonal. Hence, by the Pythagorean identity,
		$$\begin{aligned}
			\|(z-1)x+zty\|_H^2
			&=\|(z-1)x\|_H^2+\|zty\|_H^2\\
			&=|z-1|^2+t^2|z|^2\\
			&=(1+t^2)|z|^2-2\operatorname{Re}(z)+1\\
			&=(1+t^2)\left(
			|z|^2-\frac{2}{1+t^2}\operatorname{Re}(z)
			+\frac{1}{(1+t^2)^2}
			\right)
			+1-\frac{1}{1+t^2}\\
			&=(1+t^2)\left|z-\frac{1}{1+t^2}\right|^2
			+\frac{t^2}{1+t^2}.
		\end{aligned}$$
		It follows that
		\[
		\|(z-1)x+zty\|_H^2\geq\frac{t^2}{1+t^2}
		\]
		for every $z\in\C$, and equality holds if and only if $z=\frac{1}{1+t^2}$. Therefore,
		\[
		\min_{z\in\C}\|(z-1)x+zty\|_H
		=
		\sqrt{\frac{t^2}{1+t^2}}
		=
		\frac{t}{\sqrt{1+t^2}}.
		\]
	\end{proof}
	\section{Construction of the operator}\label{sec:construction}
	
	Throughout this section, we fix $\varepsilon\in(0,1)$. We shall construct a sequence of operators $(A_n)_{n\ge 1}\subset\Lop(\ell_2(\Nzero))$ and then use it to define an operator-weighted backward shift on $\Hcal$:
	$$T: (x_0,x_1, x_2, \cdots)\in \mathcal{H}\longmapsto (A_1x_1, A_2x_2, \cdots)\in \mathcal{H}.$$
	The construction is divided into four steps. We first choose the vectors involved in the construction and the parameters associated with them. We then define finite blocks of operator weights and we obtain the estimates required for our argument. Finally, we choose the lengths of these blocks recursively and arrange them consecutively.
	
	\subsection{Choice of parameters}\label{subsec:parameters}
	Set
	\[
	g:=(e_0,0,0,\ldots)\in\Hcal
	\qquad\text{and}\qquad
	L:=\operatorname{span}\{g\}.
	\]
	For every $s\in\N$, define
	\[
	\eps_s:=\varepsilon(1-2^{-s})
	\qquad\text{and}\qquad
	a_s:=
	\sqrt{\frac{\varepsilon^2-\eps_s^2}
		{1-\eps_s^2}}.
	\]
	For each $s\in\N$, consider the compact annulus
	\[
	\mathcal A_s:=\{z\in\C:2^{-s}\leq |z|\leq 2^s\}.
	\]
	By compactness, there exist $N_s\in\N$ and $c_{s,1},\ldots,c_{s,N_s}\in\mathcal A_s$ such that
	\begin{equation}\label{eq:annulus-cover}
		\mathcal A_s
		\subset
		\bigcup_{j=1}^{N_s}
		\left\{z\in\C:
		|z-c_{s,j}|<\frac{a_s}{2}|c_{s,j}|
		\right\}.
	\end{equation}
	Let 
	\[
	C_{00}(\Hcal):=
	\left\{
	(x_0,\ldots,x_N,0,\ldots)\in\Hcal:
	\begin{array}{l}
		N\in\N,\quad
		x_j\in\operatorname{span}_{\mathbb Q+i\mathbb Q}
		\{e_0,\ldots,e_N\},\quad 0\leq j\leq N
	\end{array}
	\right\}
	\]
	and
	\[
	\mathcal Q:=C_{00}(\Hcal)\setminus L.
	\]
	The set $\mathcal Q$ is countable and dense in $\Hcal$. Choose a sequence $(q_s)_{s\in\N}$ in $\mathcal Q$ such that every element of $\mathcal Q$ occurs infinitely many times and, for each $s\in\N$,
	\[
	q_s=(q_{s,0},\ldots,q_{s,s},0,\ldots),
	\qquad
	q_{s,j}\in\operatorname{span}\{e_0,\ldots,e_s\}
	\quad (0\leq j\leq s).
	\]
	In particular, every tail of $(q_s)_{s\in\N}$ is dense in $\Hcal$.
	
	Define recursively $p_1:=2$ and
	\[
	p_{s+1}:=p_s+N_s+1
	\qquad (s\in\N),
	\]
	and set
	\[
	J_s:=\{p_s,p_s+1,\ldots,p_s+N_s\}.
	\]
	The sets $(J_s)_{s\in\N}$ are pairwise disjoint and
	\[
	\bigcup_{s\in\N}J_s=\{k\in\N:k\geq2\}.
	\]
	For $s\in\N$ and $1\leq j\leq N_s$, define
	\[
	b_{s,j}:=
	\frac{1-\varepsilon^2}{1-\eps_s^2}\,c_{s,j}.
	\]
	For every $k\geq2$, let
	\[
	u^{(k)}:=
	\begin{cases}
		q_s,
		&\text{if }k=p_s\text{ for some }s\in\N,\\[1mm]
		b_{s,j}g,
		&\text{if }k=p_s+j\text{ for some }s\in\N
		\text{ and }1\leq j\leq N_s.
	\end{cases}
	\]
	
	For every $k\geq2$, define $\beta_k:=\eps_s$ whenever $k\in J_s$, and set $\beta_0:=\eps_1$. Since $\eps_s\to\varepsilon$, we have
	\begin{equation}\label{eq:beta-bounds}
		0<\beta_0\leq\beta_k<\varepsilon
		\quad (k\geq2),
		\qquad
		\beta_k\longrightarrow\varepsilon.
	\end{equation}
	
	Set $\omega_1:=1$ and $d_1:=2$. For every $k\geq2$, define
	\[
	C_k:=\frac{\beta_k}{\sqrt{1-\beta_k^2}},
	\qquad
	\omega_k:=\frac{C_k}{2},
	\qquad
	d_k:=1.
	\]
	Finally, put
	\[
	\omega_0:=
	\min\left\{1,
	\frac{\beta_0}{2\sqrt{1-\beta_0^2}}
	\right\}
	\quad\text{and}\quad
	\omega_\infty:=
	\max\left\{1,
	\frac{\varepsilon}{2\sqrt{1-\varepsilon^2}}
	\right\}.
	\]
	Since the function
	\[
	t\longmapsto\frac{t}{2\sqrt{1-t^2}}
	\]
	is strictly increasing on $(0,1)$, it follows from \eqref{eq:beta-bounds} that
	\begin{equation}\label{eq:omega-bounds}
		0<\omega_0\leq\omega_k\leq\omega_\infty<\infty
		\qquad (k\in\N).
	\end{equation}
	
	For every $k\geq2$, let
	\[
	\mathcal F_k
	:=
	\left\{
	(u_0,\ldots,u_{k-1},0,\ldots)\in\Hcal:
	u_j\in\operatorname{span}\{e_0,\ldots,e_{k-1}\},\ 0\leq j\leq k-1
	\right\}.
	\]
	Since $p_s\geq s+1$ for every $s\in\N$, we have
	\begin{equation}\label{eq:uk-in-Fk}
		u^{(k)}\in\mathcal F_k
		\qquad (k\geq2).
	\end{equation}
	\subsection{Finite weight blocks}\label{subsec:finite-blocks}
	
	We now construct finite blocks of operator weights that will be used in the definition of the operator. Define
	\[
	\lambda:=\frac{4}{\sqrt{1-\varepsilon^2}}.
	\]
	For every $k\in\N$, let
	\[
	E_k:=\operatorname{span}\{e_0,e_k\},
	\qquad
	G_k:=E_k^\perp.
	\]
	Fix $k,r\in\N$. For $1\leq j\leq r+k+1$, define
	$A_j^{(k,r)}\in\Lop(\ell_2(\Nzero))$ by
	\[
	A_j^{(k,r)}(\alpha e_0+\beta e_k+w)
	:=
	\alpha e_0+\beta v_j+\gamma_jw,
	\qquad
	\alpha,\beta\in\C,\quad w\in G_k,
	\]
	where
	\[
	v_j:=
	\begin{cases}
		e_0+\omega_ke_k,
		&j=1,\\
		2e_k,
		&2\leq j\leq r,\\
		e_k,
		&r+1\leq j\leq r+k-1,\\
		2^{-(r-1)}e_k,
		&j=r+k,\\
		-\dfrac{d_k}{\omega_k}e_0+\dfrac{d_k}{\omega_k}e_k,
		&j=r+k+1,
	\end{cases}
	\]
	and
	\[
	\gamma_j:=
	\begin{cases}
		\lambda,
		&1\leq j\leq r+k,\\
		\dfrac{d_k}{\lambda^{r+k}},
		&j=r+k+1.
	\end{cases}
	\]
	If the lower endpoint of one of these ranges exceeds its upper endpoint, the corresponding range is understood to be empty. The following elementary observation allows us to arrange these blocks consecutively.
	
	\begin{lemma}\label{lem:weights-bounded}
		The operators $A_j^{(k,r)}$, where $k,r\in\N$ and $1\leq j\leq r+k+1$, are invertible and uniformly bounded.
	\end{lemma}
	
	\begin{proof}
		Fix $k,r\in\N$. The restriction of each $A_j^{(k,r)}$ to $G_k$ is a nonzero scalar multiple of the identity and is therefore invertible. With respect to the ordered orthonormal basis $(e_0,e_k)$ of $E_k$, the only nondiagonal restrictions are represented by
		\begin{equation}\label{eq:special-matrices}
			\begin{pmatrix}
				1&1\\
				0&\omega_k
			\end{pmatrix}
			\qquad\text{and}\qquad
			\begin{pmatrix}
				1&-\dfrac{d_k}{\omega_k}\\[1mm]
				0&\dfrac{d_k}{\omega_k}
			\end{pmatrix}.
		\end{equation}
		Their determinants are $\omega_k$ and $\frac{d_k}{\omega_k}$, respectively, and hence are nonzero. All the remaining restrictions to $E_k$ are diagonal with nonzero diagonal entries. Thus, every $A_j^{(k,r)}$ is invertible.
		
		Set
		\[
		M_0:=\max\left\{2,\omega_{\infty},\frac{2}{\omega_0}\right\}.
		\]
		By \eqref{eq:omega-bounds} and the fact that $d_k\in\{1,2\}$, all entries of the matrices in \eqref{eq:special-matrices} are bounded in modulus by $M_0$. Hence, their operator norms are bounded by $2M_0$. The remaining restrictions to $E_k$ have norm at most $2$, while the restrictions to $G_k$ have norm at most $\lambda$. Therefore, the family of operators $\left( A_j^{(k,r)} \right)_{k,r,j}$ is uniformly bounded.
	\end{proof}
	
	Fix $k,r\in\N$. Define
	\[
	B_0^{(k,r)}:=I_{\ell_2(\Nzero)}
	\]
	and
	\[
	B_j^{(k,r)}:=A_1^{(k,r)}\cdots A_j^{(k,r)}
	\qquad (1\leq j\leq r+k+1).
	\]
	Their action on the canonical basis is given by the following identities.
	
	\begin{lemma}\label{lem:block-products}
		Let $k,r\in\N$. Then
		\[
		B_j^{(k,r)}e_0=e_0
		\qquad (0\leq j\leq r+k+1),
		\]
		and
		\[
		B_j^{(k,r)}e_k
		=
		\begin{cases}
			e_0+\omega_ke_k,
			&j=1,\\
			2^{j-1}(e_0+\omega_ke_k),
			&2\leq j\leq r,\\
			2^{r-1}(e_0+\omega_ke_k),
			&r+1\leq j\leq r+k-1,\\
			e_0+\omega_ke_k,
			&j=r+k,\\
			d_ke_k,
			&j=r+k+1.
		\end{cases}
		\]
		Moreover, for every $s\in\N\setminus\{k\}$,
		\[
		B_j^{(k,r)}e_s
		=
		\begin{cases}
			\lambda^je_s,
			&0\leq j\leq r+k,\\
			d_ke_s,
			&j=r+k+1.
		\end{cases}
		\]
		In particular,
		\begin{equation}\label{eq:complete-blocks}
			B_3^{(1,1)}e_0=e_0,
			\qquad
			B_3^{(1,1)}e_s=2e_s\quad(s\geq1),
			\qquad
			B_{r+k+1}^{(k,r)}=I_{\ell_2(\Nzero)}\quad(k\geq2).
		\end{equation}
	\end{lemma}
	\begin{proof}
		Every weight fixes $e_0$, and hence $B_j^{(k,r)}e_0=e_0$. Since
		\[
		B_j^{(k,r)}e_k
		=
		B_{j-1}^{(k,r)}v_j,
		\]
		the definition of $v_j$ gives successively
		\[
		\begin{aligned}
			B_1^{(k,r)}e_k&=e_0+\omega_ke_k,\\
			B_j^{(k,r)}e_k&=2^{j-1}(e_0+\omega_ke_k)
			&& (2\leq j\leq r),\\
			B_j^{(k,r)}e_k&=2^{r-1}(e_0+\omega_ke_k)
			&& (r+1\leq j\leq r+k-1),\\
			B_{r+k}^{(k,r)}e_k&=e_0+\omega_ke_k.
		\end{aligned}
		\]
		For the last weight, we obtain
		\[
		\begin{aligned}
			B_{r+k+1}^{(k,r)}e_k
			&=B_{r+k}^{(k,r)}
			\left(-\frac{d_k}{\omega_k}e_0+\frac{d_k}{\omega_k}e_k\right)\\
			&=d_ke_k.
		\end{aligned}
		\]
		If $s\in\N\setminus\{k\}$, each of the first $r+k$ weights multiplies $e_s$ by $\lambda$, whereas the last one multiplies it by $d_k/\lambda^{r+k}$. Therefore
		\[
		B_j^{(k,r)}e_s=\lambda^je_s
		\quad (0\leq j\leq r+k),
		\qquad
		B_{r+k+1}^{(k,r)}e_s=d_ke_s.
		\]
		The identities in \eqref{eq:complete-blocks} now follow from $d_1=2$ and $d_k=1$ for $k\geq2$.
	\end{proof}
	
	\subsection{Auxiliary estimates}
	Define
	$$\alpha
	:=
	\max\left\{
	\frac{\sqrt5}{4}\varepsilon,
	\frac{\varepsilon}{\sqrt{4-3\varepsilon^2}}
	\right\}.$$
	Since $0<\varepsilon<1$, it follows that $0<\alpha<\varepsilon$. The next lemma establishes the main estimate needed for the construction.
	
	\begin{lemma}\label{lem:coding}
		Let $k\geq2$ and $r_1,\ldots,r_{k-1}\in\N$, with $r_1=1$. Define $m_1:=0$ and $m_{i+1}:=m_i+r_i+i+1$ for $1\leq i<k$. For each $j\in\{0,\ldots,k-1\}$, let $h_j<k$ be the unique positive integer such that $m_{h_j}\leq j<m_{h_j+1},$ and set $l_j:=j-m_{h_j}.$ Let
		\[
		u=(u_0,\ldots,u_{k-1},0,\ldots)\in\mathcal F_k,
		\]
		and let $d>0$. Then there exists $R\in\N$ such that, for every $r\geq R$, there are vectors $v_{0}^{(k,r)},\ldots,v_{k-1}^{(k,r)}\in\ell_2(\Nzero)$ satisfying the following conditions:
		
		\begin{enumerate}
			\item[\rm (i)] Let $v^{(k,r)}\in\Hcal$ be defined by
			\[
			v^{(k,r)}:=(0,\cdots, 0, \underset{\text{position } m_k+r}{\underbrace{v_{0}^{(k,r)}}},\ldots,v_{k-1}^{(k,r)}, 0, \cdots).
			\]
			Then $\|v^{(k,r)}\|\leq d.$
			
			\item[\rm (ii)] For $0\leq j\leq k-1$, define
			\[
			z_j^{(k,r)}:=
			\begin{cases}
				\left( B_{l_j}^{(1,1)} \right)^{-1}
				B_3^{(1,1)}B_{r+j}^{(k,r)}v_{j}^{(k,r)},
				&h_j=1,\\[1mm]
				\left( B_{l_j}^{(h_j,r_{h_j})} \right)^{-1}
				B_{r+j}^{(k,r)}v_{j}^{(k,r)},
				&h_j\geq2.
			\end{cases}
			\]
			Then $\|z_0^{(k,r)}-u_0\|_2=\beta_k|u_{0,0}|$, and $\|z_j^{(k,r)}-u_j\|_2\leq\alpha\|u_j\|_2$ for $1\leq j\leq k-1$, where $u_{0,0}$ denotes the $0^{\text{th}}$ coordinate of $u_0$.
			
			\item[\rm (iii)] If $z^{(k,r)}:=(z_0^{(k,r)},\ldots,z_{k-1}^{(k,r)},0,\ldots),$ then
			\begin{equation}\label{eq:coding-error}
				\|z^{(k,r)}-u\|^2
				\leq
				\beta_k^2|u_{0,0}|^2
				+\alpha^2\bigl(\|u\|^2-|u_{0,0}|^2\bigr).
			\end{equation}
		\end{enumerate}
	\end{lemma}
	\begin{proof}
		Write
		\[
		u_j=\sum_{s=0}^{k-1}u_{j,s}e_s
		\qquad (0\leq j\leq k-1).
		\]
		Fix $r\in \N$ and $0\leq j\leq k-1$, and put $h=h_j$ and $l=l_j$. We define $v_{j}^{(k,r)}$ according to whether $l=0$ or $l>0$. Assume first that $l=0$, and set
		\[
		v_{j}^{(k,r)}
		:=
		\sum_{s=1}^{k-1}
		\frac{u_{j,s}}{d_h\lambda^{r+j}}e_s
		+
		\frac{u_{j,0}}
		{2^{r-1}(1+d_h^2\omega_k^2)}e_k.
		\]
		Using Lemma~\ref{lem:block-products}, we obtain
		\[
		z_j^{(k,r)}
		=
		\sum_{s=1}^{k-1}u_{j,s}e_s
		+\frac{u_{j,0}}{1+d_h^2\omega_k^2}e_0
		+\frac{d_h\omega_k}{1+d_h^2\omega_k^2}u_{j,0}e_k,
		\]
		and hence
		\begin{equation}\label{eq:error-l-zero}
			\|z_j^{(k,r)}-u_j\|_2
			=
			\frac{d_h\omega_k}{\sqrt{1+d_h^2\omega_k^2}}
			|u_{j,0}|.
		\end{equation}
		For $j=0$, we have $h=1$, and therefore $d_h=2$ and
		$2\omega_k=C_k$. Consequently,
		\begin{equation}\label{eqaux2}
			\|z_0^{(k,r)}-u_0\|_2
			=
			\frac{C_k}{\sqrt{1+C_k^2}}|u_{0,0}|
			=
			\beta_k|u_{0,0}|.
		\end{equation}
		If $j\geq1$, then $h\geq2$, so $d_h=1$. Since
		$\omega_k=C_k/2$, \eqref{eq:error-l-zero} gives
		\begin{equation}\label{eqaux3}
			\|z_j^{(k,r)}-u_j\|_2
			=
			\frac{\beta_k}{\sqrt{4-3\beta_k^2}}|u_{j,0}|
			\leq
			\frac{\varepsilon}{\sqrt{4-3\varepsilon^2}}\|u_j\|_2
			\leq
			\alpha\|u_j\|_2.
		\end{equation}
		Assume now that $l>0$. By Lemma~\ref{lem:block-products}, there exists
		$c>0$ such that
		\[
		B_l^{(h,r_h)}e_h=c(e_0+\omega_he_h),
		\qquad
		c\leq2^l.
		\]
		Moreover,
		\[
		\left(B_l^{(h,r_h)}\right)^{-1}e_h
		=
		-\frac{1}{\omega_h}e_0
		+\frac{1}{c\omega_h}e_h,
		\]
		while
		\[
		\left(B_l^{(h,r_h)}\right)^{-1}e_s
		=
		\lambda^{-l}e_s
		\qquad (s\geq1,\ s\neq h).
		\]
		Define
		\[
		v_{j}^{(k,r)}
		:=
		\sum_{\substack{1\leq s\leq k-1\\s\neq h}}
		\frac{\lambda^l}{d_h\lambda^{r+j}}u_{j,s}e_s
		+
		\frac{c\omega_h}{d_h\lambda^{r+j}}u_{j,h}e_h
		+
		\frac{u_{j,0}+cu_{j,h}}{2^{r-1}}e_k.
		\]
		By Lemma~\ref{lem:block-products} we obtain
		\[
		z_j^{(k,r)}
		=
		u_j+
		\frac{d_h\omega_k}{\lambda^l}
		(u_{j,0}+cu_{j,h})e_k.
		\]
		Since
		\[
		d_h\omega_k\leq C_k
		<
		\frac{\varepsilon}{\sqrt{1-\varepsilon^2}},
		\qquad
		c\leq2^l,
		\]
		and $l\geq1$, we obtain
		\begin{equation}\label{eqaux4}
			\begin{aligned}
				\|z_j^{(k,r)}-u_j\|_2
				&\leq
				\frac{\varepsilon}{\sqrt{1-\varepsilon^2}}
				\left(
				\frac{|u_{j,0}|}{\lambda^l}
				+\frac{2^l|u_{j,h}|}{\lambda^l}
				\right)\\
				&\leq
				\frac{\varepsilon}{4}
				\bigl(|u_{j,0}|+2|u_{j,h}|\bigr)\\
				&\leq
				\frac{\sqrt5}{4}\varepsilon\|u_j\|_2
				\leq
				\alpha\|u_j\|_2.
			\end{aligned}
		\end{equation}
		
		For each $j$, the coefficients defining $v_j^{(k,r)}$ tend to zero as
		$r\to\infty$. Hence
		\[
		\|v^{(k,r)}\|^2
		=
		\sum_{j=0}^{k-1}\|v_j^{(k,r)}\|_2^2
		\longrightarrow0.
		\]
		Therefore, there exists $R\in\N$ such that
		\[
		\|v^{(k,r)}\|\leq d
		\qquad (r\geq R),
		\]
		which proves (i). Moreover, \eqref{eqaux2}-\eqref{eqaux4} prove (ii). Finally,
		\[
		\begin{aligned}
			\|z^{(k,r)}-u\|^2
			&=\sum_{j=0}^{k-1}\|z_j^{(k,r)}-u_j\|_2^2\\
			&\leq
			\beta_k^2|u_{0,0}|^2
			+\alpha^2\bigl(\|u\|^2-|u_{0,0}|^2\bigr),
		\end{aligned}
		\]
		and hence (iii) follows.
	\end{proof}
	
	\begin{remark}
		From now on, whenever we refer to the vector $z^{(k,r)}$ associated with $u$ in Lemma~\ref{lem:coding}, we mean the vector obtained from the construction given in the proof of that lemma.
	\end{remark}
	
	\begin{lemma}\label{lem:strict-margin}
		Under the assumptions and notation of Lemma~\ref{lem:coding}, suppose that
		$u\notin L$. Then, for every $r\geq R$,
		\[
		\varepsilon^2\|u\|^2-\|z^{(k,r)}-u\|^2
		\geq
		(\varepsilon^2-\alpha^2)\dist(u,L)^2.
		\]
		Moreover, if
		\[
		D(u):=
		\frac{(\varepsilon^2-\alpha^2)\dist(u,L)^2}
		{2\varepsilon\|u\|},
		\]
		then, for every $r\geq R$,
		\[
		\|z^{(k,r)}-u\|
		\leq
		\varepsilon\|u\|-D(u).
		\]
	\end{lemma}
	\begin{proof}
		Since $L=\operatorname{span}\{g\}$ and $g=(e_0,0,\ldots)$, we have
		\begin{equation}\label{eq:distance-L}
			\dist(u,L)^2
			=
			\|u\|^2-|u_{0,0}|^2.
		\end{equation}
		By \eqref{eq:coding-error} and the fact that $\beta_k<\varepsilon$,
		\[
		\begin{aligned}
			\|z^{(k,r)}-u\|^2
			&\leq
			\beta_k^2|u_{0,0}|^2
			+\alpha^2\bigl(\|u\|^2-|u_{0,0}|^2\bigr)\\
			&\leq
			\varepsilon^2|u_{0,0}|^2
			+\alpha^2\bigl(\|u\|^2-|u_{0,0}|^2\bigr).
		\end{aligned}
		\]
		Together with \eqref{eq:distance-L}, this yields
		\[
		\varepsilon^2\|u\|^2-\|z^{(k,r)}-u\|^2
		\geq
		(\varepsilon^2-\alpha^2)\dist(u,L)^2.
		\]
		
		Now, by the definition of $D(u)$,
		\[
		\|z^{(k,r)}-u\|^2
		\leq
		\varepsilon^2\|u\|^2
		-2\varepsilon\|u\|D(u)
		\leq
		\bigl(\varepsilon\|u\|-D(u)\bigr)^2.
		\]
		Since $u\notin L$, we have $D(u)>0$. Moreover,
		\[
		D(u)
		\leq
		\frac{\varepsilon^2\|u\|^2}{2\varepsilon\|u\|}
		=
		\frac{\varepsilon}{2}\|u\|,
		\]
		and hence $\varepsilon\|u\|-D(u)>0$. Taking square roots gives
		\[
		\|z^{(k,r)}-u\|
		\leq
		\varepsilon\|u\|-D(u).
		\]
	\end{proof}
	For vectors in $L$, the vector $z^{(k,r)}$ can be computed explicitly.
	\begin{lemma}\label{lem:line-case}
		Under the assumptions and notation of Lemma~\ref{lem:coding}, suppose that
		\[
		u=ag
		\]
		for some $a\in\C\setminus\{0\}$. Then, for every $r\geq R$,
		\[
		z^{(k,r)}
		=
		\bigl((1-\beta_k^2)ae_0
		+(1-\beta_k^2)C_kae_k,0,0,\ldots\bigr).
		\]
		Moreover, for every $b\in\C$,
		\[\|z^{(k,r)}-bg\|^2-\varepsilon^2|b|^2=(1-\varepsilon^2)\left|b-\frac{1-\beta_k^2}{1-\varepsilon^2}a\right|^2-\frac{(\varepsilon^2-\beta_k^2)(1-\beta_k^2)}{1-\varepsilon^2}|a|^2.\]
	\end{lemma}
	\begin{proof}
		Since $u=ag$, we have
		\[
		u_0=ae_0
		\qquad\text{and}\qquad
		u_j=0\quad (1\leq j\leq k-1).
		\]
		In the construction given in the proof of Lemma~\ref{lem:coding}, we have
		$h_0=1$ and $l_0=0$. Hence
		\[
		z_0^{(k,r)}
		=
		\frac{a}{1+4\omega_k^2}e_0
		+
		\frac{2\omega_k}{1+4\omega_k^2}ae_k.
		\]
		Since $2\omega_k=C_k$ and
		\[
		C_k^2=\frac{\beta_k^2}{1-\beta_k^2},
		\]
		we have
		\[
		\frac{1}{1+C_k^2}=1-\beta_k^2.
		\]
		Therefore,
		\[
		z_0^{(k,r)}
		=
		(1-\beta_k^2)ae_0
		+
		(1-\beta_k^2)C_kae_k.
		\]
		Moreover, the construction gives $v_j^{(k,r)}=0$, and hence
		$z_j^{(k,r)}=0$, for every $1\leq j\leq k-1$. Thus
		\[
		z^{(k,r)}
		=
		\bigl((1-\beta_k^2)ae_0
		+(1-\beta_k^2)C_kae_k,0,0,\ldots\bigr).
		\]
		
		Now fix $b\in\C$. Using again
		\[
		C_k^2=\frac{\beta_k^2}{1-\beta_k^2},
		\]
		we obtain
		\[
		\begin{aligned}
			\|z^{(k,r)}-bg\|^2
			&=
			\left|(1-\beta_k^2)a-b\right|^2
			+(1-\beta_k^2)^2C_k^2|a|^2\\
			&=
			|b|^2
			-2(1-\beta_k^2)\operatorname{Re}(b\overline a)
			+(1-\beta_k^2)|a|^2.
		\end{aligned}
		\]
		Consequently,
		\[
		\begin{aligned}
			\|z^{(k,r)}-bg\|^2-\varepsilon^2|b|^2
			&=
			(1-\varepsilon^2)|b|^2
			-2(1-\beta_k^2)\operatorname{Re}(b\overline a)
			+(1-\beta_k^2)|a|^2\\
			&=
			(1-\varepsilon^2)
			\left|b-\frac{1-\beta_k^2}{1-\varepsilon^2}a\right|^2
			-\frac{(\varepsilon^2-\beta_k^2)(1-\beta_k^2)}
			{1-\varepsilon^2}|a|^2,
		\end{aligned}
		\]
		as desired.
	\end{proof}
	For $k\geq2$, define
	\[
	K_k:=
	\begin{cases}
		\displaystyle
		\left\{y\in\Hcal:
		\|y-q_s\|\leq\frac{D(q_s)}{2(1+\varepsilon)}
		\right\},
		&\text{if }k=p_s\text{ for some }s\in\N,\\[4mm]
		\displaystyle
		\left\{bg:
		|b-c_{s,j}|\leq\frac{a_s}{2}|c_{s,j}|
		\right\},
		&\text{if }k=p_s+j\text{ for some }s\in\N
		\text{ and }1\leq j\leq N_s.
	\end{cases}
	\]
	\begin{lemma}\label{lem:cover}
		We have
		\[
		\Hcal\setminus\{0\}=\bigcup_{k\geq2}K_k.
		\]
	\end{lemma}
	\begin{proof}
		First, $0\notin K_k$ for every $k\geq2$. Indeed, suppose first that $k=p_s$ for some $s\in\N$. Since
		\[
		\frac{D(q_s)}{2(1+\varepsilon)}
		\leq
		\frac{\varepsilon}{4(1+\varepsilon)}\|q_s\|
		<\|q_s\|,
		\]
		it follows that $0\notin K_{p_s}$. Now, suppose that $k=p_s+j$ for some $s\in\N$ and $1\leq j\leq N_s$. Since $c_{s,j}\neq0$ and $a_s<1$, we have $0\notin K_{p_s+j}$. Now, let $y\in\Hcal\setminus\{0\}$. Assume first that $y=bg\in L\setminus\{0\}$ for some $b\in\C\setminus\{0\}$. Choose $s\in\N$ such that
		\[
		2^{-s}\leq|b|\leq2^s.
		\]
		Then $b\in\mathcal A_s$. By \eqref{eq:annulus-cover}, there exists $j\in\{1,\ldots,N_s\}$ such that
		\[
		|b-c_{s,j}|<\frac{a_s}{2}|c_{s,j}|,
		\]
		and therefore $y\in K_{p_s+j}$. Assume now that $y\notin L$. Since $(q_s)_{s\in\N}$ is dense in $\Hcal$, there exists a strictly increasing sequence $(s_n)_{n\in\N}$ such that $q_{s_n}\to y.$ For all sufficiently large $n$,
		\[
		\dist(q_{s_n},L)\geq\frac{\dist(y,L)}{2}>0
		\qquad\text{and}\qquad
		\|q_{s_n}\|\leq2\|y\|.
		\]
		Consequently,
		\[
		\frac{D(q_{s_n})}{2(1+\varepsilon)}
		\geq
		\frac{(\varepsilon^2-\alpha^2)\dist(y,L)^2}
		{32\varepsilon(1+\varepsilon)\|y\|}.
		\]
		Thus, for all sufficiently large $n$,
		\[
		\|q_{s_n}-y\|
		<
		\frac{(\varepsilon^2-\alpha^2)\dist(y,L)^2}
		{32\varepsilon(1+\varepsilon)\|y\|}
		\leq
		\frac{D(q_{s_n})}{2(1+\varepsilon)}.
		\]
		Therefore, $y\in K_{p_{s_n}}$.
	\end{proof}
	\subsection{Construction of the weights}\label{subsec:weights}
	We are now ready to construct the operator that will be used in the proof of our main theorem. The construction relies on the finite weight blocks and the estimates established in the previous subsections. We start by defining, for $k\ge 2$, 
	\[
	s_k:=
	\begin{cases}
		\displaystyle
		\frac{D(q_s)}{2},
		&\text{if }k=p_s\text{ for some }s\in\N,\\[3mm]
		\displaystyle
		\frac{3(1-\varepsilon^2)a_s^2|c_{s,j}|}
		{8\varepsilon(1+a_s/2)},
		&\text{if }k=p_s+j\text{ for some }s\in\N
		\text{ and }1\leq j\leq N_s.
	\end{cases}
	\]
	By Lemma~\ref{lem:weights-bounded}, we can fix $M\geq1$ such that
	\begin{equation}\label{eq:block-bound}
		\|A_j^{(k,r)}\|\leq M
	\end{equation}
	for every $k,r\in\N$ and $1\leq j\leq r+k+1$. We recursively construct a sequence $(r_k)_{k\geq1}$ and vectors $(v^{(k)})_{k\geq2}$. Set
	\[
	r_1:=1,
	\qquad
	m_1:=0,
	\qquad
	m_2:=3.
	\]
	Suppose that, for some $k\geq2$, the integers $r_1,\ldots,r_{k-1}$ and $m_1,\ldots,m_k$ have already been defined. For $2\leq i<k$, set $n_i:=m_i+r_i$, and define
	\begin{equation}\label{eq:tk}
		t_k:=
		\begin{cases}
			2^{-2},
			&k=2,\\[1mm]
			\displaystyle
			2^{-k}\min\{1,s_i M^{-n_i}:2\leq i<k\},
			&k\geq3.
		\end{cases}
	\end{equation}
	As observed in \eqref{eq:uk-in-Fk}, we have $u^{(k)}\in\mathcal F_k$. Applying Lemma~\ref{lem:coding} with $u=u^{(k)}$ and $d=t_k$, we obtain $R\in\N$ such that the conclusions of that lemma hold for every $r\geq R$. Choose $r_k\geq R$, set
	\[
	v^{(k)}:=v^{(k,r_k)},
	\]
	and define
	\begin{equation}\label{eq:nk-mk}
		n_k:=m_k+r_k,
		\qquad
		m_{k+1}:=m_k+r_k+k+1.
	\end{equation}
	In particular,
	\begin{equation}\label{eq:vk-small}
		\|v^{(k)}\|\leq t_k\leq2^{-k}.
	\end{equation}
	Proceeding recursively, we obtain the sequences $(r_k)_{k\geq1}$, $(m_k)_{k\geq1}$, $(n_k)_{k\geq2}$, and $(v^{(k)})_{k\geq2}$. 
	
	Now, for every $k\in\N$ and $1\leq j\leq r_k+k+1$, define
	\[
	A_{m_k+j}:=A_j^{(k,r_k)}.
	\]
	Since
	\[
	m_{k+1}=m_k+r_k+k+1,
	\]
	this defines a sequence $(A_n)_{n\geq1}$ on all of $\N$. Each $A_n$ is invertible and, by \eqref{eq:block-bound},
	\begin{equation}\label{eq:global-weight-bound}
		\sup_{n\geq1}\|A_n\|\leq M.
	\end{equation}
	For each $k\in\N$, we call
	\[
	A_{m_k+1},\ldots,A_{m_{k+1}}
	\]
	the $k$-th block of weights. By construction,
	\[
	A_{m_k+j}=A_j^{(k,r_k)}
	\qquad (1\leq j\leq r_k+k+1),
	\]
	and therefore the product of the weights in the $k$-th block is
	\[
	A_{m_k+1}\cdots A_{m_{k+1}}
	=
	B_{r_k+k+1}^{(k,r_k)}.
	\]
	Moreover, by \eqref{eq:complete-blocks}, for every $k\geq2$, the product of the weights in the $k$-th block is the identity. Therefore, for every $k\geq2$,
	\begin{equation}\label{eq:product-before-block}
		A_1\cdots A_{m_k}=B_3^{(1,1)}.
	\end{equation}
	In particular, using again \eqref{eq:complete-blocks},
	\begin{equation}\label{eq:product-action}
		A_1\cdots A_{m_k}e_0=e_0,
		\qquad
		A_1\cdots A_{m_k}e_s=2e_s
		\quad (s\geq1).
	\end{equation}
	
	\section{Main result}\label{sec:main-result}
	The sequence of weights constructed above defines the operator that gives the desired endpoint behavior.
	
	\begin{proof}[Proof of Theorem~A]
		Let $(A_n)_{n\geq1}$ be the sequence constructed in Subsection~\ref{subsec:weights}, and define the operator-weighted backward shift
		\[
		T(x_0,x_1,\ldots)
		:=
		(A_1x_1,A_2x_2,\ldots),
		\qquad
		(x_n)_{n\in\Nzero}\in\Hcal.
		\]
		By \eqref{eq:global-weight-bound},
		\[
		\begin{aligned}
			\|Tx\|^2
			&=\sum_{n=0}^{\infty}\|A_{n+1}x_{n+1}\|_2^2\\
			&\leq M^2\sum_{n=0}^{\infty}\|x_{n+1}\|_2^2
			\leq M^2\|x\|^2.
		\end{aligned}
		\]
		Thus $T$ is bounded.  Moreover, for every $n,j\in\Nzero$,
		\begin{equation}\label{eq:iterate-formula}
			(T^nx)_j
			=
			A_{j+1}\cdots A_{j+n}x_{j+n},
		\end{equation}
		where the product is understood to be the identity when $n=0$.
		
		We first show that $T$ is not $\delta$-hypercyclic for any $\delta<\varepsilon$. Fix $0<\delta<\varepsilon$ and an arbitrary vector
		\[
		u=(u_n)_{n\in\Nzero}\in\Hcal.
		\]
		Choose $\delta'\in(\delta,\varepsilon)$. By \eqref{eq:beta-bounds}, there exists $k_0\geq2$ such that
		\begin{equation}\label{eq:beta-delta-prime}
			\beta_k>\delta'
			\qquad (k\geq k_0).
		\end{equation}
		For every $n\in\Nzero$, write
		\[
		u_n=a_ne_0+w_n,
		\qquad
		a_n\in\C,
		\qquad
		w_n\perp e_0.
		\]
		Since $\|u_n\|_2\to0$, by increasing $k_0$ if necessary, we may assume that
		\begin{equation}\label{eq:an-bound}
			|a_n|\leq1
			\qquad (n\geq m_{k_0}).
		\end{equation}
		
		Let $n>m_{k_0}$, and let $k\geq k_0$ be the unique integer such that
		\[
		m_k<n\leq m_{k+1}.
		\]
		Write $n=m_k+j$, where $1\leq j\leq r_k+k+1$.
		
		Assume first that $j\leq r_k+k$. By Lemma~\ref{lem:block-products}, there exist $c\in\C$ and $w\in G_k$ such that
		\[
		B_j^{(k,r_k)}w_n
		=
		ce_0+c\omega_ke_k+w.
		\]
		Using \eqref{eq:product-action} and $2\omega_k=C_k$, we obtain
		\[
		A_1\cdots A_nu_n
		=
		(a_n+c)e_0+cC_ke_k+2w.
		\]
		Since $w$ is orthogonal to both $e_0$ and $e_k$, Lemma~\ref{lem:hilbert-estimate}, applied with $t=C_k$, gives, for every $R>1$,
		\begin{equation}\label{eq:lower-bound-inside-block}
			\begin{aligned}
				\|Re_0-A_1\cdots A_nu_n\|_2
				&\geq
				\|(R-a_n-c)e_0-cC_ke_k\|_2\\
				&\geq
				\frac{C_k}{\sqrt{1+C_k^2}}|R-a_n|\\
				&=
				\beta_k|R-a_n|.
			\end{aligned}
		\end{equation}
		
		If $j=r_k+k+1$, then, by \eqref{eq:complete-blocks},
		\[
		B_{r_k+k+1}^{(k,r_k)}=I_{\ell_2(\Nzero)}.
		\]
		Hence the $e_0$-coordinate of $A_1\cdots A_nu_n$ is $a_n$, and therefore
		\begin{equation}\label{eq:lower-bound-end-block}
			\|Re_0-A_1\cdots A_nu_n\|_2
			\geq
			|R-a_n|.
		\end{equation}
		It follows from \eqref{eq:beta-delta-prime}-\eqref{eq:lower-bound-end-block} that, in both cases,
		\begin{equation}\label{eq:uniform-lower-bound}
			\|Re_0-A_1\cdots A_nu_n\|_2
			\geq
			\delta'|R-a_n|
			\geq
			\delta'(R-1).
		\end{equation}
		
		Choose $R>1$ sufficiently large that
		\[
		\delta'(R-1)>\delta R
		\]
		and
		\[
		R-\|A_1\cdots A_nu_n\|_2>\delta R
		\qquad (0\leq n\leq m_{k_0}),
		\]
		where the product is understood to be the identity when $n=0$. Set $y:=(Re_0,0,0,\ldots)$. By \eqref{eq:iterate-formula} and \eqref{eq:uniform-lower-bound}, for every $n\in\Nzero$,
		\[
		\begin{aligned}
			\|T^nu-y\|
			&\geq
			\|(T^nu)_0-Re_0\|_2\\
			&=
			\|A_1\cdots A_nu_n-Re_0\|_2\\
			&>\delta R
			=\delta\|y\|.
		\end{aligned}
		\]
		Thus $u$ is not $\delta$-hypercyclic. Since $u\in\Hcal$ was arbitrary, $T$ is not $\delta$-hypercyclic. Since $\delta\in(0,\varepsilon)$ was arbitrary, $T$ is not $\delta$-hypercyclic for any $\delta<\varepsilon$. We now prove that $T$ is $\varepsilon$-hypercyclic. Define
		\[
		x:=\sum_{k=2}^{\infty}v^{(k)}.
		\]
		By \eqref{eq:vk-small}, the series converges in $\Hcal$. Fix $y\in\Hcal\setminus\{0\}$. By Lemma~\ref{lem:cover}, choose $i\geq2$ such that $y\in K_i.$ If $2\leq k<i$, then $v^{(k)}$ is supported in the outer coordinates $n_k,\ldots,n_k+k-1.$ By \eqref{eq:nk-mk},
		\[
		n_k+k-1=m_{k+1}-2<n_i,
		\]
		and therefore
		\begin{equation}\label{eq:previous-packets-vanish}
			T^{n_i}v^{(k)}=0
			\qquad (2\leq k<i).
		\end{equation}
		
		For $0\leq j\leq i-1$, let $h_j$ and $l_j$ be as in Lemma~\ref{lem:coding}. Since the $(n_i+j)$-th coordinate of $v^{(i)}$ is $v_j^{(i,r_i)}$, it follows from \eqref{eq:complete-blocks}, \eqref{eq:product-before-block}, \eqref{eq:iterate-formula}, and the definition of the weights that
		\[
		(T^{n_i}v^{(i)})_j
		=
		\begin{cases}
			\left(B_{l_j}^{(1,1)}\right)^{-1}
			B_3^{(1,1)}B_{r_i+j}^{(i,r_i)}v_j^{(i,r_i)},
			&h_j=1,\\[1mm]
			\left(B_{l_j}^{(h_j,r_{h_j})}\right)^{-1}
			B_{r_i+j}^{(i,r_i)}v_j^{(i,r_i)},
			&h_j\geq2.
		\end{cases}
		\]
		Hence
		\[
		T^{n_i}v^{(i)}=z^{(i,r_i)},
		\]
		where $z^{(i,r_i)}$ is the vector given by Lemma~\ref{lem:coding} with $u=u^{(i)}$. Assume first that $i=p_s$ for some $s\in\N$. Then $u^{(i)}=q_s$ and, since $y\in K_i$,
		\[
		\|y-q_s\|
		\leq
		\frac{D(q_s)}{2(1+\varepsilon)}.
		\]
		By Lemma~\ref{lem:strict-margin},
		\[
		\|T^{n_i}v^{(i)}-q_s\|
		\leq
		\varepsilon\|q_s\|-D(q_s).
		\]
		Therefore,
		\[
		\begin{aligned}
			\|T^{n_i}v^{(i)}-y\|
			&\leq
			\varepsilon\|q_s\|-D(q_s)+\|q_s-y\|\\
			&\leq
			\varepsilon\|y\|+(1+\varepsilon)\|q_s-y\|-D(q_s)\\
			&\leq
			\varepsilon\|y\|-\frac{D(q_s)}{2}\\
			&=
			\varepsilon\|y\|-s_i.
		\end{aligned}
		\]
		
		Assume now that $i=p_s+\ell$ for some $s\in\N$ and $1\leq\ell\leq N_s$. Then $u^{(i)}=b_{s,\ell}g$ and $y=bg$ for some $b\in\C$ satisfying
		\[
		|b-c_{s,\ell}|
		\leq
		\frac{a_s}{2}|c_{s,\ell}|.
		\]
		Since $\beta_i=\eps_s$ and
		\[
		c_{s,\ell}=\frac{1-\beta_i^2}{1-\varepsilon^2}b_{s,\ell},
		\]
		Lemma~\ref{lem:line-case} gives
		\[
		\begin{aligned}
			\varepsilon^2|b|^2-\|T^{n_i}v^{(i)}-bg\|^2
			&=
			(1-\varepsilon^2)
			\left(a_s^2|c_{s,\ell}|^2-|b-c_{s,\ell}|^2\right)\\
			&\geq
			\frac34(1-\varepsilon^2)a_s^2|c_{s,\ell}|^2.
		\end{aligned}
		\]
		In particular, $\|T^{n_i}v^{(i)}-bg\|\leq\varepsilon|b|.$ Moreover,
		\[
		|b|
		\leq
		\left(1+\frac{a_s}{2}\right)|c_{s,\ell}|.
		\]
		Consequently,
		\[
		\begin{aligned}
			\varepsilon|b|-\|T^{n_i}v^{(i)}-bg\|
			&=
			\frac{\varepsilon^2|b|^2-\|T^{n_i}v^{(i)}-bg\|^2}
			{\varepsilon|b|+\|T^{n_i}v^{(i)}-bg\|}\\
			&\geq
			\frac{3(1-\varepsilon^2)a_s^2|c_{s,\ell}|}
			{8\varepsilon(1+a_s/2)}=s_i.
		\end{aligned}
		\]
		Thus, in both cases,
		\begin{equation}\label{eq:packet-margin}
			\|T^{n_i}v^{(i)}-y\|
			\leq
			\varepsilon\|y\|-s_i.
		\end{equation}
		
		For every $k>i$, the definition of $t_k$ in \eqref{eq:tk} gives $\|v^{(k)}\|\leq2^{-k}s_iM^{-n_i}.$ Since $\|T\|\leq M$,
		\begin{equation}\label{eq:future-packets}
			\left\|
			\sum_{k>i}T^{n_i}v^{(k)}
			\right\|
			\leq
			\sum_{k>i}M^{n_i}\|v^{(k)}\|
			\leq
			s_i\sum_{k>i}2^{-k}
			<s_i.
		\end{equation}
		Using \eqref{eq:previous-packets-vanish}-\eqref{eq:future-packets}, we obtain
		\[
		\begin{aligned}
			\|T^{n_i}x-y\|
			\leq
			\|T^{n_i}v^{(i)}-y\|
			+
			\left\|\sum_{k>i}T^{n_i}v^{(k)}\right\|<\varepsilon\|y\|.
		\end{aligned}
		\]
		Since $y\neq0$ was arbitrary, $x$ is an $\varepsilon$-hypercyclic vector for $T$. Hence $T$ is $\delta$-hypercyclic for every $\delta\in[\varepsilon,1)$. Together with the first part of the proof, this completes the proof.
	\end{proof}
	
	\section*{Acknowledgements}

The first named author was supported by FAPEMIG grant APQ-02622-26. The second named author was supported by CNPq grants 406457/2023-9, 403964/2024-5 and 305807/2025-0. The third named author was financed in part by the Fundação Carlos Chagas Filho de Amparo à Pesquisa do Estado do Rio de Janeiro (FAPERJ) – Brazil - `Nota 10' Doctoral Fellowship.
	
\section*{Declaration on the use of generative AI}

During the preparation of this manuscript, the authors used ChatGPT (OpenAI) as a supplementary tool for language editing, organization of the exposition, consistency checks, literature exploration, and exploratory discussion of proof strategies. All mathematical arguments and results were reviewed and independently verified by the authors, who take full responsibility for the content of the manuscript.

\section*{Competing interests}

The authors declare that they have no competing interests.


	\smallskip
	
	{\footnotesize
		
		\bigskip\noindent
		{\sc Vinícius V. Fávaro}
		
		\smallskip\noindent
		Departamento de Matemática, Universidade Federal de Uberlândia,
		Uberlândia, Brazil.\\
		\textit{ e-mail address}: vvfavaro@ufu.br
		
		\bigskip\noindent {\sc Daniel Pellegrino}
		
		\smallskip\noindent
		Departamento de Matemática, Universidade Federal da Paraíba,
		João Pessoa, Brazil.\\
		\textit{ e-mail address}: dmpellegrino@gmail.com
		
		\bigskip\noindent
		{\sc Jo\~ao V. A. Pinto}
		
		\smallskip\noindent
		Departamento de Matem\'atica Aplicada, Instituto de Matem\'atica, Universidade Federal do Rio de Janeiro, Brazil.\\
		\textit{ e-mail address}: pinto.j.v.a@gmail.com
		
		\bigskip\noindent
		{\sc Geivison Ribeiro}
		
		\smallskip\noindent
		{Departamento de Matemática, Universidade Federal de Sergipe,
			Itabaiana, Brazil.\\
			\textit{ e-mail address}: geivison@mat.ufs.br
			
		}
	\end{document}